%% file: main.tex
\UseRawInputEncoding
\documentclass{amsart}
\input{preamble.tex}

\usepackage{xparse}
\makeatletter
\RenewDocumentCommand{\title}{om}{%
   \IfNoValueTF{#1}
     {\gdef\shorttitle{Non-vanishing of Dirichlet L-functions at the central point}}
     {\gdef\shorttitle{#1}}%
   \gdef\@title{#2}%
}
\makeatother

\title{Non-Vanishing of Dirichlet L-functions at the central point with restricted root number} 
\author{Adam Earnst}
\date{March 23, 2026}
\address{Department of Mathematics\\
         Rutgers University,
         Piscataway, NJ 08854\\
         USA}
\email{adam.earnst@rutgers.edu}

\begin{document}

\begin{abstract}
    We prove asymptotics for mollified first and second moments of subfamilies of Dirichlet $L$-functions given by shrinking angular restrictions on the root number.
    Using these moments, we prove that for even primitive characters with prime conductor $q$, a positive proportion of the central values $L(1/2,\chi)$ do not vanish as $q\to\infty$.
\end{abstract}

\maketitle

\tableofcontents

\section{Introduction}
It is a well-known conjecture, often attributed to Chowla, that the value of a primitive Dirichlet $L$-function,
\begin{equation*}
    L(s,\chi)=\sum_{n=1}^\infty \frac{\chi(n)}{n^s},
\end{equation*}
at the central point $s=1/2$ is always nonzero. There is an abundance of literature on the non-vanishing of $L$-functions at the central point. Most of these results rely on computing mollified moments of Dirichlet $L$-functions. 
Balasubramanian and Murty \cite{BM92} used mollifiers to compute that a positive proportion $(\geq 0.04)$ of Dirichlet $L$-functions with modulus $q$ vanish at the central point. Iwaniec and Sarnak \cite{IS99} improved on this result using simpler methods, obtaining a proportion of non-vanishing of $1/3$. The best known result for general moduli is a proportion of $0.3411$, due to Bui \cite{Bui12}. For prime moduli, Khan, Mili\'cevi\'c, and Ngo \cite{KMN22} proved that $5/13$ of primitive Dirichlet $L$-functions do not vanish at the central point.

For a primitive even Dirichlet character, we define the root number as the normalized Gauss sum,
\begin{equation*}
    \epsilon(\chi)=\frac{1}{\sqrt{q}}\sum_{t\shortmod q}\chi(t)e\left(\frac{t}{q}\right),
\end{equation*}
which has modulus $1$. We define the angle of the root number, $\theta_\chi$, to be the unique value in $\R/\Z$ satisfying $\epsilon(\chi)=e(\theta_\chi)$. The root number can encode information about the behavior of an $L$-function at the central point. An example is that a self-dual $L$-function with root number $-1$ trivially vanishes at the central point due to its functional equation.

Due to this connection between the root number $\epsilon(\chi)$ and the central $L$-value $L(1/2,\chi)$, we consider the problem of non-vanishing of Dirichlet $L$-functions subject to angular restrictions of the root number. In particular, we use the mollifier technique to obtain results on non-vanishing of the central value of $L$-functions in this restricted family. 
\subsection{Statement of Results}
Let $q$ be a prime and let $I_q\subset \R/\Z$ be an interval of length $\mu(I_q)$. Then, for $\epsilon>0$, we define
\begin{equation*}
    N(q,I_q,\epsilon)=\#\left\{\chi\shortmod{q}\mid \chi(-1)=1,\ \theta_\chi\in I_q,\ |L(1/2,\chi)|>\epsilon\cdot\mu(I_q)(\log q)^{-1/2} \right\}.
\end{equation*}
\begin{thm}\label{thm: positive proportion}
    Let $\{I_q\}$ be a sequence of intervals varying over $q$ prime.
    Suppose that $\mu(I_q)\geq C q^{-\eta}$ for some $0\leq\eta<1/480$. Then, for $q$ sufficiently large depending on $C$, $\eta$, and $\epsilon$,
    \begin{equation*}
        \frac{1}{\mu(I_q)\varphi^+(q)}N(q,I_q,\epsilon)\geq (1-\epsilon)c(\eta),
    \end{equation*}
    where $\varphi^+(q)$ denotes the number of primitive even characters modulo $q$ and $c(\eta)$ is positive.
\end{thm}
\begin{rmk}\label{rmk: choice of c(eta)}
    Theorem \ref{thm: positive proportion} gives a proportion of non-vanishing on of at least $(1-\epsilon)c(\eta)$ by the equidistribution of Gauss sums proved in \cite{K88}. Our proof allows us to take
    \begin{equation*}
        c(\eta)= \frac{1}{25}-\frac{96}{5}\eta.
    \end{equation*}
    In particular, for a fixed interval $I$, we have a proportion of non-vanishing of $\frac{1}{25}(1-\epsilon)$.
\end{rmk}
To obtain this non-vanishing result we evaluate mollified first and second moments weighted by smooth functions of the angle $\theta_\chi$. We apply the mollifier from \cite{IS99}, which takes the form
\begin{equation*}
M(\chi)=\sum_{m\leq M}\frac{x_m\chi(m)}{m^{1/2}},
\end{equation*}
where the coefficients $x_m$ are defined in \eqref{mollifier coefficients}, and the length is taken to be $M=q^\alpha$ for $\alpha>0$.
Consider a family of smooth functions $f_q:\R/\Z\to\C$ varying over $q$ prime. Let $J$ be a sufficiently large integer (we prove that $J=20$ is large enough) and suppose that the family $f_q$ satisfies the bounds
\begin{align}
    \|f_q^{(J)}\|_1&\ll q^{1/24-\alpha-\epsilon}\left|\int_0^1 f_q(x)\,dx\right|,\label{f_q derivative condition}\\ 
    \|f_q\|_1&\ll \left|\int_0^1 f_q(x)\,dx\right|.\label{f_q condition 2}
\end{align}
Note that these bounds are trivially satisfied if $f_q=f$ is a fixed, positive-valued function. Then we have the following theorem for the smoothed moments.
\begin{thm}\label{thm: smoothed mollified moments} Let $f_q:\R/\Z\to\C$ be a family of smooth functions varying over $q$ prime, satisfying conditions \eqref{f_q derivative condition} and \eqref{f_q condition 2}. Then, as $q\to\infty$ over $q$ prime, we have the following asymptotics for the smoothed mollified first and second moments:
    \begin{align*}
        &\cC=\evensum f_q(\theta_\chi)M(\chi)L(1/2,\chi)= \left(\int_0^1 f_q(x)\,dx\right)\left\{\varphi^+(q)+o\left(\varphi^+(q)\right)\right\},\\
        &\cD=\evensum f_q(\theta_\chi)|M(\chi)L(1/2,\chi)|^2= \left(\int_0^1 f_q(x)\,dx\right)\left\{\left(1+\frac{1}{\alpha}\right)\varphi^+(q)+o\left(\varphi^+(q)\right)\right\},
    \end{align*}
    where $M(\chi)$ is a mollifier of length $q^\alpha$ with $\alpha<1/24$.
\end{thm}
\begin{rmk}
    Throughout this paper, $\sum^*$ denotes summation over primitive characters, and $\sum^+$ denotes summation over even primitive characters.
\end{rmk}
To obtain the smoothed mollified moments in Theorem \ref{thm: smoothed mollified moments}, we first estimate first and second moments weighted by a character and a power of the root number.
\begin{thm}\label{thm: first weighted moment}
    Let $q,\ m$ be positive integers and $k$ any integer. Then 
    \begin{align*}
        \cA(m,k):=\evensum \chi(m)\epsilon(\chi)^k L(1/2,\chi)=
        \begin{cases}
            \displaystyle \delta(m)\varphi^+(q)+O\left(\tau(q)\sqrt{mq}\right),\quad &k=0,\\[5pt]
            \displaystyle \frac{\varphi^+(q)}{\sqrt{m}}+O\left(\tau(q)\sqrt{q}\right),\quad &k=-1,\\[10pt]
            \displaystyle O\left(|k|^{\omega(q)}\tau(q)q^{3/4}\right),\quad &k\neq0,-1,
        \end{cases}
    \end{align*}
    where $\delta(m)$ is the Kronecker delta symbol
    \begin{align*}
        \delta(m)=
        \begin{cases}
            1,\quad &m=1,\\
            0,\quad &\textrm{otherwise},
        \end{cases}
    \end{align*}
    and the implied constants are absolute.
\end{thm}

\begin{thm}\label{thm: second weighted moment}
    Let $q$ be a prime, $k$ an integer with $|k|\geq 2$, and $m_1,\,m_2$ integers. Then
    \begin{equation*}
        \cB(m_1,m_2,k):=\evensum \chi(m_1)\ol{\chi}(m_2)\epsilon(\chi)^k|L(1/2,\chi)|^2\ll |k|^{18}q^{23/24+\epsilon},
    \end{equation*}
    with any $\epsilon>0$, where the implied constant depends only on $\epsilon$.
\end{thm}

\begin{rmk}
    The approach used to compute the weighted second moments in Theorem \ref{thm: second weighted moment} produces error terms for $|k|\leq 1$ which cannot be bounded well enough individually. For example, the case of $k=\pm1$, $m_1=m_2$ contributes $O(q)$. However, the explicit choice of a mollifier, which exploits the cancellation of the M\"obius function, allows us to obtain logarithmic saving in the mollified second moment.
\end{rmk}

\subsection{Remarks}
Evaluating the mollified moments of the family of Dirichlet $L$-functions of fixed modulus $q$ with restrictions on the root number gives rise to hyper-Kloosterman sums
\begin{equation*}
    \Kl_k(x;q)=\frac{1}{q^{(k-1)/2}}\sum_{\substack{x_1,\ldots,x_k\shortmod{q}\\x_1\cdots x_k\equiv x\shortmod{q}}}e\left(\frac{x_1+\cdots+x_k}{q}\right).
\end{equation*}
Due to Deligne, $|\Kl_k(x;q)|\leq k^{\omega(q)}$ for all $(x,q)=1$. In the evaluation of the mollified second moment using an approximate functional equation, bilinear forms with Kloosterman sums of the shape
\begin{equation*}
    \sum_{m\leq M}\sum_{n\leq N}f(m)g(n)\Kl_k(cm^an^b;q)
\end{equation*}
appear, where $a,b\in\{\pm1\}$, and we need extra saving over Deligne's bound.
We can consider an asymmetrical approximate functional equation so that the sums with $(a,b)=(1,1)$ have length $>q$, the sums with $(a,b)=(-1,-1)$ have length $<q$, and the sums with $(a,b)=(1,-1)$ and $(a,b)=(-1,1)$ have length $q$.
For $a=b=1$, work of Fouvry, Kowalski, and Michel \cite{FKM14} gives significant power saving over the trivial bound. For the cases $(a,b)=(1,-1)$ and $(a,b)=(-1,1)$, we take advantage of the asymmetric approximate functional equation and apply a shifting trick along the lines of Kowalski, Michel, and Sawin in \cite{KMS17} to reduce to a complete exponential sum.

A result similar to Theorem \ref{thm: positive proportion} was achieved independently by Berta and zur Verth \cite{BzV26}, which considered simultaneous non-vanishing in toroidal families of Dirichlet $L$-functions with root number restrictions. This result depends upon recent work by Fouvry, Kowalski, Michel, and Sawin \cite{FKMS25}, which gives non-trivial bounds for a very general class of bilinear forms with trace functions, where the dependence on the complexity of the associated $\ell$-adic sheaf is not known to be polynomial. Our approach differs in its use of a $+hn$ shift to the $m$-variable in the monomial kernels $K(m\ol{n})$, allowing us to apply the earlier results in \cite{FKM14} which have polynomial dependence on the conductor. In turn, this polynomial dependence allows us to obtain a positive proportion of non-vanishing for subfamilies of Dirichlet $L$-functions where the angle $\theta_\chi$ is restricted to shrinking intervals as $q\to\infty$.

\subsection{Acknowledgments}
I thank Henryk Iwaniec for 
his many ideas, suggestions, and discussions throughout the writing of this manuscript. I would also like to thank Matthew Young and Philippe Michel for their comments and encouragement. This material is based upon work supported by the National Science Foundation Graduate
Research Fellowship Program under Grant No. 2233066. Any opinions, findings,
and conclusions or recommendations expressed in this material are those of the author
and do not necessarily reflect the views of the National Science Foundation.

\section{Setup}
We analyze restricted families of Dirichlet $L$-functions by their root number. Associated to a primitive character $\chi$ mod $q$ is the root number,
\begin{equation*}
    \epsilon(\chi)=\frac{\tau(\chi)}{i^\delta\sqrt{q}},
\end{equation*}
where $\tau(\chi)$ is the standard Gauss sum and
\begin{equation*}
    \delta=\frac{1-\chi(-1)}{2}.
\end{equation*}
This root number appears in the functional equation for the Dirichlet $L$-function, which we recall here.
\begin{prop}
    Let $\chi\pmod{q}$ be a primitive Dirichlet character. We define the completed Dirichlet $L$-function as
    \begin{equation*}
        \Lambda(s,\chi)=q^{s/2}\gamma(s,\chi)L(s,\chi),
    \end{equation*}
    where the gamma factor $\gamma(s,\chi)$ is given by
    \begin{equation*}
        \gamma(s,\chi)=\pi^{-s/2}\Gamma\left(\frac{s+\delta}{2}\right).
    \end{equation*}
    Then, $\Lambda(s,\chi)$ satisfies the functional equation
    \begin{equation*}
        \Lambda(s,\chi)=\epsilon(\chi)\Lambda(1-s,\ol{\chi}).
    \end{equation*}
\end{prop}
The root number has absolute value $1$, and can thus be uniquely written as
\begin{equation*}
    \epsilon(\chi)=e(\theta_\chi)
\end{equation*}
for $\theta_\chi\in\R/\Z$.
We are interested in estimating mollified moments with restrictions on the root number
\begin{equation}\label{eq: restricted moments}
    \evensumrest M(\chi)L(1/2,\chi),\qquad \evensumrest|M(\chi)L(1/2,\chi)|^2.
\end{equation}
We can smooth out this problem, letting $f:\R/\Z\to\C$ be a smooth function, and considering the smoothed moments
\begin{equation*}
    \evensum f(\theta_\chi) M(\chi)L(1/2,\chi),\qquad\evensum f(\theta_\chi)|M(\chi)L(1/2,\chi)|^2,
\end{equation*}
which approximate the restricted moments in \eqref{eq: restricted moments} if $f$ approximates the indicator function of an interval. By the Fourier expansion of $f$, 
\begin{equation*}
    f(\theta)=\sum_{k\in\Z} c_ke(\theta)^k,
\end{equation*}
the study of these smoothed moments reduces to studying mollified moments twisted by powers of the root number,
\begin{equation*}\label{weighted moments equation*}
    \evensum\epsilon(\chi)^k M(\chi)L(1/2,\chi),\qquad \evensum\epsilon(\chi)^k|M(\chi)L(1/2,\chi)|^2.
\end{equation*}
A handy tool for computing these mollified moments is the approximate functional equation.
\begin{thm}[Approximate functional equation, \cite{IK04}, Theorem 5.3]
Let $\chi\negthickspace\pmod{q}$ be a primitive Dirichlet character. Let $G(u)$ be a holomorphic function bounded on vertical strips, with the normalization $G(0)=1$. Then at the central point $s=1/2$,
\begin{equation*}
    L(1/2,\chi)=\sum_{n=1}^\infty \frac{\chi(n)}{\sqrt{n}}V\left(\frac{n}{X\sqrt{q}}\right)+\epsilon(\chi)\sum_{n=1}^\infty\frac{\ol{\chi}(n)}{\sqrt{n}}V\left(\frac{nX}{\sqrt{q}}\right),
\end{equation*}
where $V(y)$ is a smooth function defined by
\begin{equation*}
    V(y)=\frac{1}{2\pi i}\int_{(2)}y^{-u}G(u)\frac{\gamma(1/2+u,\chi)}{\gamma(1/2,\chi)}\,\frac{du}{u}.
\end{equation*}
\end{thm}

\begin{rmk}
    For the analysis in this paper, we choose $G(u)$ to be of the form $P(u)\exp(u^2)$, where $P(u)$ vanishes at the poles of $\gamma(1/2+u,\chi)$ with $|\Re(u)|\leq 2$. With this choice, we get the following bounds for $V(y)$ and its derivatives.
\end{rmk}

\begin{prop}[\cite{IK04}, Proposition 5.4]
Let $V(y)$ be as defined above. Then, for any $A>0$, the derivatives of $V(y)$ satisfy
    \begin{align}
        y^aV_s^{(a)}(y)&\ll_A (1+y)^{-A},\nonumber \\
        y^aV_s^{(a)}(y)&=\delta_a+O(y^{2}),\label{eq: bounds on V}
    \end{align}
    where $\delta_0=1$ and $\delta_a=0$ otherwise.
\end{prop}

The definition of the root number $\epsilon(\chi)$ and the smooth function $V(y)$ depend on whether $\chi$ is an even or odd character. Thus, when evaluating moments over the full family of primitive characters $\chi\pmod{q}$, we will take moments over the even and odd characters separately. In this paper, we will provide proofs for the subfamily of even primitive characters; the arguments for odd primitive characters proceed in a near-identical manner. Since we want to consider even primitive characters separately, we need a version of orthogonality of characters over this subfamily.
\begin{prop}[\cite{IK04}, Chapter 3]\label{prop: orthogonality of characters}
Let $q$ be a positive integer $m\in\Z$ such that $(m,q)=1$. Then
\begin{align*}
    \primsum \chi(m)&=\sum_{\substack{vw=q\\m\equiv1\shortmod{w}}}\mu(v)\varphi(w),\\
    \evensum \chi(m)&=\frac{1}{2}\sum_{\substack{vw=q\\m\equiv1\shortmod{w}}}\mu(v)\varphi(w)+\frac{1}{2}\sum_{\substack{vw=q\\m\equiv-1\shortmod{w}}}\mu(v)\varphi(w).
\end{align*}
\end{prop}
Throughout the paper, we will let
\begin{equation*}
    \varphi^*(q)=\sum_{vw=q}\mu(v)\varphi(w)
\end{equation*}
denote the number of primitive characters mod $q$. We let $\varphi^+(q)$ denote the number of primitive even characters. From Proposition \ref{prop: orthogonality of characters}, we see that $\varphi^+(q)=\frac{1}{2}\varphi^*(q)+O(1)$.
\begin{rmk}
    In the version of orthogonality of characters for even primitive characters, there are two sums, one over $m\equiv1\pmod{w}$ and another over $m\equiv-1\pmod{w}$. For simplicity of notation, we write
    \begin{equation*}
        \evensum \chi(m)=\frac{1}{2}\sum_{\substack{vw=q\\m\equiv\pm 1\shortmod{w}}}\mu(v)\varphi(w).
    \end{equation*}
\end{rmk}

\section{Weighted moments}
In this section, we prove Theorems \ref{thm: first weighted moment} and \ref{thm: second weighted moment}, which provide estimates for the first and second moments weighted by characters and powers of the root number.
\subsection{Weighted first moments}
We begin by evaluating estimates for the weighted first moment
\begin{equation*}
\cA(m,k)=\evensum\chi(m)\epsilon(\chi)^kL(1/2,\chi),
\end{equation*}
where $(m,q)=1$ and $m<q$.
We can split $\cA(m,k)$ into two sums using the approximate functional equation,
\begin{align*}
\cA(m,k)&=\sum_{n=1}^\infty \frac{1}{\sqrt{n}}V\left(\frac{n}{X\sqrt{q}}\right)\evensum \chi(m)\chi(n)\epsilon(\chi)^k+\sum_{n=1}^\infty \frac{1}{\sqrt{n}}V\left(\frac{nX}{\sqrt{q}}\right)\evensum \chi(m)\ol{\chi}(n)\epsilon(\chi)^{k+1}\\
&=:\cA_1(m,k)+\cA_2(m,k),
\end{align*}
where $X>0$ will be chosen later to depend on $k$.
We will first evaluate the sums where the root number disappears, $\cA_1(m,0)$ and $\cA_2(m,-1)$. By orthogonality of characters, we get that
\begin{align*}
    \cA_1(m,0)&=\frac{1}{2}\sum_{vw=q}\mu(v)\varphi(w)\underset{mn\equiv\pm1\shortmod{w}}{\primsumempty}\frac{1}{\sqrt{n}}V\left(\frac{n}{X\sqrt{q}}\right)\\
    &=\frac{1}{2}\delta(m)\varphi^*(q)V\left(\frac{1}{X\sqrt{q}}\right)+ \frac{1}{2}\sum_{vw=q}\mu(v)\varphi(w)\sum_{\ell=1}^\infty\underset{{mn=\ell w\pm 1}}{\primsumempty}\frac{1}{\sqrt{n}}V\left(\frac{n}{X\sqrt{q}}\right),
\end{align*}
where $\delta(m)$ is the Kronecker delta symbol,
\begin{align*}
    \delta(m)=
    \begin{cases}
        1,\quad &m=1,\\
        0,\quad&\textrm{otherwise}.
    \end{cases}
\end{align*}
We use $\primsumempty$ to denote that the summation is over integers coprime to $q$.
Since $(m,q)=1$, we also have $(m,w)=1$, so if there exists a solution to $mn=\ell w+1$ and $mn'=\ell' w+1$, then $m\mid (\ell-\ell')$. Using the fact that $V(y)\ll (1+y)^{-1}$, a decreasing function, we have that our remaining sum is bounded by
\begin{align*}
&\sum_{w\mid q}w\sum_{\ell=0}^\infty\frac{\sqrt{m}}{\sqrt{(1+m\ell)w}}\left(1+\frac{(1+m\ell)w}{mX\sqrt{q}}\right)^{-1}\ll \sum_{w\mid q}\left(\sqrt{wm}+\sum_{\ell=1}^\infty\frac{\sqrt{w}}{\sqrt{\ell}}\left(1+\frac{\ell w}{X\sqrt{q}}\right)^{-1}\right).
\end{align*}
Splitting the sum into $\ell\leq X\sqrt{q}/w$ and $\ell>X\sqrt{q}/w$, we have
\begin{equation*}
    \cA_1(m,0)=\delta(m)\varphi^+(q)+O\left(\frac{q^{3/4}}{X^{1/2}}+\sqrt{qm}\tau(q)+q^{1/4}X^{1/2}\tau(q)\right).
\end{equation*}
Evaluating $\cA_2(m,-1)$ is similar, and we obtain
\begin{equation*}
    \cA_2(m,-1)=\frac{\varphi^+(q)}{\sqrt{m}}+O\left(q^{3/4}X^{1/2}+\frac{q^{1/4}}{X^{1/2}}\tau(q)\right).
\end{equation*}
We now evaluate the sums $\cA_1(m,k)$ and $\cA_2(m,k)$ when the root numbers do not disappear. Let's first  evaluate $\cA_1(m,k)$ for $k>0$. Expanding $\epsilon(\chi)$ into a Gauss sum, we get
\begin{align*}
    \cA_1(m,k)&=\frac{1}{2q^{k/2}}\sum_{vw=q}\mu(v)\varphi(w)\underset{n\ }{\primsumempty}\frac{1}{\sqrt{n}}V\left(\frac{n}{X\sqrt{q}}\right)\underset{\substack{t_1,\ldots,t_k\\t_1\cdots t_k\equiv\pm\ol{mn}\shortmod{w}}}{\primsumempty}e\left(\frac{t_1+\cdots+t_k}{q}\right)\nonumber\\
    &=\frac{1}{2\sqrt{q}}\sum_{vw=q}\mu(v)\varphi(w)\underset{n\ }{\primsumempty}\frac{1}{\sqrt{n}}V\left(\frac{n}{X\sqrt{q}}\right)\underset{\substack{t\shortmod{q}\\t\equiv\pm\ol{mn}\shortmod{w}}}{\primsumempty}\Kl_k(t;q)
    \label{equation with sum of Kloosterman sums}.
\end{align*}
The sum over $t$ has $\varphi(q)/\varphi(w)$ terms. We then use the Deligne bound for hyper-Kloosterman sums, which gives that $|\Kl_k(t;q)|\leq k^{\omega(q)}$. Thus, our sum is bounded by
\begin{equation*}
    \cA_1(m,k)
    \ll k^{\omega(q)}\tau(q)X^{1/2}q^{3/4}.
\end{equation*}
Applying the same steps yields an identical bound for $k<0$. For the dual sums $\cA_2(m,k)$, $k\neq-1$, we get
\begin{equation*}
    \cA_2(m,k)\ll |k+1|^{\omega(q)}\tau(q)X^{-1/2}q^{3/4}.
\end{equation*}
To minimize the error terms, we take $X=q^{1/2}$ when $k=0$, $X=q^{-1/2}$ when $k=-1$, and $X=1$ otherwise. This completes the proof of Theorem \ref{thm: first weighted moment}.

\subsection{Weighted second moments}
In this subsection, we prove Theorem \ref{thm: second weighted moment}. We want to evaluate the second moment weighted by $\chi(m_1)\ol{\chi}(m_2)\epsilon(\chi)^k$:
\begin{equation*}
    \cB(m_1,m_2,k)=\evensum \chi(m_1)\ol{\chi}(m_2)\epsilon(\chi)^k|L(1/2,\chi)|^2.
\end{equation*}
The case $k=0$ reduces to sums considered by Iwaniec and Sarnak in \cite{IS99}.
They obtain the bound
\begin{equation*}
    \cB(m_1,m_2,0)=\frac{\varphi^+(q)\varphi(q)}{q\sqrt{m_1m_2}}
\log\frac{L^2}{m_1m_2}+E,
\end{equation*}
where $L$ is defined by
\begin{equation*}
    \log L=\frac{1}{2}\log\frac{q}{\pi}+\frac{1}{2}\psi\left(\frac{1}{4}\right)+\gamma+\eta(q),
\end{equation*}
with $\psi(s)=\Gamma'(s)/\Gamma(s)$ and
\begin{equation*}
    \eta(q)=\sum_{p\mid q}\frac{\log p}{p-1}.
\end{equation*}
The term $E$ denotes the error term, which is provided in \cite{IS99}. This error terms cannot be bounded well individually, but summation over the variables $m_1$ and $m_2$ introduced by the mollifier will allow for a good bound on average.

Now, assume that $k\geq2$, and $q$ is prime. The methods used below for $k\geq2$ will produce identical bounds for $k\leq -2$. We apply an asymmetric approximate functional equation for $L(1/2,\chi)$ and $L(1/2,\ol{\chi})$. This gives
\begin{align*}
    L(1/2,\chi)&=\sum_n\frac{\chi(n)}{\sqrt{n}}V\left(\frac{nX}{\sqrt{q}}\right)+\epsilon(\chi)\sum_n\frac{\ol{\chi}(n)}{\sqrt{n}}V\left(\frac{n}{X\sqrt{q}}\right),\\
    L(1/2,\ol{\chi})&=\sum_n\frac{\ol{\chi}(n)}{\sqrt{n}}V\left(\frac{n}{X\sqrt{q}}\right)+\epsilon(\ol\chi)\sum_n\frac{{\chi}(n)}{\sqrt{n}}V\left(\frac{nX}{\sqrt{q}}\right),
\end{align*}
where $X=q^\theta$ is to be chosen later.
Substituting these expressions into $\cB(m_1,m_2,k)$ and applying orthogonality of characters gives
\begin{equation}\label{AFE sum decomp}
    \cB(m_1,m_2,k)=\cB_1(m_1,m_2,k)+\cB_2(m_1,m_2,k)+\cB_3(m_1,m_2,k)+\cB_4(m_1,m_2,k)+O\left(k(X+X^{-1})q^{1/2+\epsilon}\right),
\end{equation}
where
\begin{align*}
    \cB_1(m_1,m_2,k)&=\frac{\varphi(q)}{2\sqrt{q}}\sum_{n_1}\sum_{n_2}\frac{1}{\sqrt{n_1n_2}}V\left(\frac{n_1X}{\sqrt{q}}\right)V\left(\frac{n_2X}{\sqrt{q}}\right)\Kl_{k-1}(\pm \ol{m_1}m_2\ol{n_1}\ol{n_2};q),\\
    \cB_2(m_1,m_2,k)&=\frac{\varphi(q)}{2\sqrt{q}}\sum_{n_1}\sum_{n_2}\frac{1}{\sqrt{n_1n_2}}V\left(\frac{n_1X}{\sqrt{q}}\right)V\left(\frac{n_2}{X\sqrt{q}}\right)\Kl_{k}(\pm \ol{m_1}m_2\ol{n_1}n_2;q),\\
    \cB_3(m_1,m_2,k)&=\frac{\varphi(q)}{2\sqrt{q}}\sum_{n_1}\sum_{n_2}\frac{1}{\sqrt{n_1n_2}}V\left(\frac{n_1}{X\sqrt{q}}\right)V\left(\frac{n_2X}{\sqrt{q}}\right)\Kl_{k}(\pm \ol{m_1}m_2{n_1}\ol{n_2};q),\\
    \cB_4(m_1,m_2,k)&=\frac{\varphi(q)}{2\sqrt{q}}\sum_{n_1}\sum_{n_2}\frac{1}{\sqrt{n_1n_2}}V\left(\frac{n_1}{X\sqrt{q}}\right)V\left(\frac{n_2}{X\sqrt{q}}\right)\Kl_{k+1}(\pm \ol{m_1}m_2{n_1}{n_2};q),
\end{align*}
where $\Kl_k(\pm x;q)$ denotes $\Kl_k(x;q)+\Kl_k(-x;q)$.
In each of the $\cB_i(m_1,m_2,k)$, the argument of the hyper-Kloosterman sum depends on either $n_1$ or its inverse, and the same for $n_2$. We will call variables that appear in the argument of the hyper-Kloosterman sum with an inverse ``inverted variables" and those that appear without an inverse ``non-inverted variables".

To bound these sums, we take advantage of work done by Fouvry, Kowalski, and Michel bounding bilinear forms with Kloosterman sums. In \cite{FKM14}, the authors obtain a power-saving bound for smoothed sums with trace weights. There, they consider smoothing functions $V(x)$ supported on $[1/2,\,2]$ satisfying
\begin{equation}\label{eq: smoothing condition}
    x^j V^{(j)}(x)\ll_j Q_V^j.
\end{equation}
For convenience, we restate their result on smoothed (type $I_2$) sums.
\begin{thm}[\cite{FKM14}, Theorem 1.16]\label{thm: FKM smoothed sums}
Let $K$ be an isotypic trace weight associated to an $\ell$-adic sheaf $\cF$ modulo $q$ prime. Let $N_1,N_2$ be positive such that $X/4\leq N_1N_2\leq X$. Let $U$, $V$, and $W$ be smooth functions supported on $[1/2,\,2]$ satisfying condition \eqref{eq: smoothing condition} with respective parameters $Q_U$, $Q_V$, and $Q_W$ all $\geq1$. Then, we have
\begin{equation}\label{FKM smoothed bound}
    \sum_{n_1,n_2}K(n_1n_2)\Big(\frac{n_1}{n_2}\Big)^{it}U\Big(\frac{n_1}{N_1}\Big)V\Big(\frac{n_2}{N_2}\Big)W\Big(\frac{n_1n_2}{X}\Big)\ll\cond(\cF)^{A}(1+|t|)^{B}(Q_U+Q_V)^{C}Q_W X\Big(1+\frac{q}{X}\Big)^{1/2}q^{-\eta},
\end{equation}
for any $\eta<1/8$, where $B,C\geq 1$ are constants depending only on $\eta$ and $A$ is absolute. The implicit constant depends only on $\eta$ and the implicit constants in \eqref{eq: smoothing condition}.
    
\end{thm}

\begin{rmk}[Explicit conductor dependence]
    In the proof of \cite[Theorem 1.16]{FKM14}, the conductor dependence arises from bounds on the amplified second moments of the sums
    \begin{equation*}
        S_{V,X}(it,k)=\sum_{n}K(n)d_{it}(n)V\left(\frac{n}{X}\right),
    \end{equation*}
    where $d_{it}$ denotes the twisted divisor function
    \begin{equation*}
        d_{u}(n)=\sum_{ab=n}\left(\frac{a}{b}\right)^u.
    \end{equation*}
    A bound on these amplified second moments is given in \cite[Proposition 4.1]{FKM15_cond}. This bound is proportional to $M^3$, where $M\geq 1$ is chosen such that $K$ is $(q,M)$-good, following the notation of \cite[Definition 1.8]{FKM15_cond}. The proof of \cite[Theorem 1.14]{FKM15_cond} implies that trace functions are $(q,M)$-good for $M\ll \cond(\cF)^{6}$. Tracing this conductor dependence through the proof, we can take $A=18$ in \eqref{FKM smoothed bound}. For our scenario, we will take $\cF=\mathcal{K}\ell_k$ to be a hyper-Kloosterman sheaf, for which $\cond(\mathcal{K}\ell_k)=k+3$.
\end{rmk}
If we take $X=q^\theta>1$, the effective lengths of the sums over non-inverted variables is longer than the sums over the inverted variables. This way, we can take advantage of the power-saving bound \eqref{FKM smoothed bound}. We now prove bounds for the four terms in equation \eqref{AFE sum decomp}.
\begin{prop}\label{bounds on B1,B2,B3,B4}
    Let $q$ be a prime, $k\geq 2$ an integer, and $X=q^\theta$ with $0<\theta\leq 1/4$, we have
    \begin{align*}
        \cB_1(m_1,m_2,k)&\ll kq^{1-\theta+\epsilon},\\
        \cB_2(m_1,m_2,k)&\ll kq^{1-\theta/2+\epsilon},\\ \cB_3(m_1,m_2,k)&\ll kq^{1-\theta/2+\epsilon},\\
        \cB_4(m_1,m_2,k)&\ll k^{18}q^{7/8+\theta+\epsilon},
    \end{align*}
    where the implicit constant depends only on $\epsilon$.
\end{prop}
When $\theta=1/12$, these bounds are balanced, implying Theorem \ref{thm: second weighted moment}.
The remainder of this subsection is devoted to proving Proposition \ref{bounds on B1,B2,B3,B4}. Let $X=q^\theta$ for some fixed $\theta>0$. The trivial bound on $\cB_1(m_1,m_2,k)$ gives
\begin{equation*}
    \cB_1(m_1,m_2,k)\ll q^{1-\theta+\epsilon}.
\end{equation*}

To analyze $\cB_4(m_1,m_2,k)$, we note that the main contribution arises from $n_1,n_2\leq Xq^{1/2+\epsilon/2}$, and the remaining terms can be bounded by an arbitrarily small power of $q$ using bounds the on $V(y)$ in \eqref{eq: bounds on V}, with $A$ sufficiently large depending on $\epsilon$.
We now apply the bound \eqref{FKM smoothed bound} to main part of $\cB_4(m_1,m_2,k)$ by splitting the sum into dyadic segments. Consider a partition of unity
\begin{equation*}
    \mathds{1}_{n>0}=\sum_{j\in\Z}\rho\left(\left(\frac{2}{3}\right)^j n\right),
\end{equation*}
where $\rho$ is supported on $[3/4,\,4/3]$. Then, we have the identity
\begin{equation*}
    1=\sum_{j=0}^{J}\rho\left(\left(\frac{2}{3}\right)^j n\right),
\end{equation*}
for all $n\leq Xq^{1/2+\epsilon/2}$ , where $J\geq\log(Xq^{1/2+\epsilon/2})/\log(3/2)$. Then, the main part of $\cB_4(m_1,m_2,k)$ reduces to
\begin{equation}\label{full smoothed dyadic sum}
    \sqrt{q}\underset{\substack{N_1=(3/2)^i\\ N_2=(3/2)^j}}{\sum\sum}\frac{1}{\sqrt{N_1N_2}}\sum_{n_1}\sum_{n_2}\frac{\sqrt{N_1}}{\sqrt{n_1}}\rho\left(\frac{n_1}{N_1}\right)V\left(\frac{n_1}{X\sqrt{q}}\right)\frac{\sqrt{N_2}}{\sqrt{n_2}}\rho\left(\frac{n_2}{N_2}\right)V\left(\frac{n_2}{X\sqrt{q}}\right)\Kl_{k+1}(cn_1n_2;q),
\end{equation}
where $i$ and $j$ range over $0\leq i,j\leq J$. Since $\rho(n_1/N_1)\rho(n_2/N_2)$ is supported on $\frac{n_1n_2}{N_1N_2}\in[9/16,\,16/9]$, we can choose a bump function $W$ supported on $[1/2,\,2]$ such that $W(x)=1$ on $[9/16,\,16/9]$. Then, inserting a factor of $W(\frac{n_1n_2}{N_1N_2})$ into our sum does not alter its value. Thus, our sum over $n_1$ and $n_2$ is given by
\begin{equation}\label{sum over n1,n2}
\sum_{n_1}\sum_{n_2}U_1\left(\frac{n_1}{N_1}\right)U_2\left(\frac{n_2}{N_2}\right)W\left(\frac{n_1n_2}{N_1N_2}\right)\Kl_{k+1}(cn_1n_2;q),
\end{equation}
where
\begin{equation*}
    U_i(x)=\frac{1}{\sqrt{x}}\rho(x)V\left(\frac{N_i}{X\sqrt{q}}x\right).
\end{equation*}
By the bounds on $V$ from \eqref{eq: bounds on V}, condition \eqref{eq: smoothing condition} is satisfied with $Q_{U_1}=Q_{U_2}=Q_W=1$.
Applying Theorem \ref{thm: FKM smoothed sums} with $\cF=\mathcal{K}\ell_k$, the sum \eqref{sum over n1,n2} is be bounded by
\begin{equation*}
    k^{18}(N_1N_2)\left(1+\frac{q}{N_1N_2}\right)^{1/2}q^{-\eta}.
\end{equation*}
Therefore, \eqref{full smoothed dyadic sum} is bounded by
\begin{equation*}
    k^{18}\sqrt{q}\underset{\substack{N_1=(3/2)^i\\ N_2=(3/2)^j}}{\sum\sum}(N_1N_2+q)^{1/2}q^{-\eta}\ll k^{18}q^{7/8+\theta+\epsilon},
\end{equation*}
upon taking $\eta=\frac{1}{8}-\frac{\epsilon}{4}$.

It remains to bound $\cB_2$ and $\cB_3$, which are identical to each other up to a change of variables. The effective length of summation for these sums is approximately $q$, and we will use a shifting trick to obtain a small power saving. Without loss of generality, let's analyze $\cB_3(m_1,m_2,k)$.
Let $N_1=X\sqrt{q}=q^{1/2+\theta}$ and $N_2=\sqrt{q}/X=q^{1/2-\theta}$, so
\begin{equation*}
    \cB_3(m_1,m_2,k)=\sqrt{q}\sum_{n_1}\sum_{n_2}\frac{1}{\sqrt{n_1n_2}}V\left(\frac{n_1}{N_1}\right)V\left(\frac{n_2}{N_2}\right)\Kl_k(cn_1\ol{n_2};q),
\end{equation*}
where $c\equiv\pm\ol{m_1}m_2\pmod{q}$.
Now, consider the shifted sum
\begin{equation*}
    \cB_{3,h}(m_1,m_2,k)=\sqrt{q}\sum_{n_1}\sum_{n_2}\frac{1}{\sqrt{n_1n_2}}V\left(\frac{n_1}{N_1}\right)V\left(\frac{n_2}{N_2}\right)\Kl_k(c(n_1+hn_2)\ol{n_2};q),
\end{equation*}
where $0\leq h\leq H\leq N_1/N_2=q^{2\theta}$.
Denoting the difference $\cB_{3,h}(m_1,m_2,k)-\cB_3(m_1,m_2,k)$ by $\Delta$, we see that
\begin{align*}
    \Delta=&-\sqrt{q}\sum_{n_2}\frac{1}{\sqrt{n_2}}V\left(\frac{n_2}{N_2}\right)\sum_{n_1\leq hn_2}\frac{1}{\sqrt{n_1}}V\left(\frac{n_1}{N_1}\right)\Kl_k(cn_1\ol{n_2};q)\\
    &+\sqrt{q}\sum_{n_2}\frac{1}{\sqrt{n_2}}V\left(\frac{n_2}{N_2}\right)\sum_{n_1>hn_2}\left(\frac{1}{\sqrt{n_1-hn_2}}V\left(\frac{n_1-hn_2}{N_1}\right)-\frac{1}{\sqrt{n_1}}V\left(\frac{n_1}{N_1}\right)\right)\Kl_k(cn_1\ol{n_2};q).
\end{align*}
We can bound the first term by $kq^{1/2}H^{1/2}N_2^{1+\epsilon}$ using the weak bound $V(y)\ll 1$. For the second term, call it $\Delta_2$, we must be more careful, using bounds for $V(y)$ and $V'(y)$. Using the bound $\Kl_k(x)\ll k$, we get that
\begin{equation*}
    \Delta_2\ll k\sqrt{q}\sum_{n_2}\frac{1}{\sqrt{n_2}}\left|V\left(\frac{n_2}{N_2}\right)\right|\sum_{n_1>hn_2}\left|\frac{1}{\sqrt{n_1-hn_2}}V\left(\frac{n_1-hn_2}{N_1}\right)-\frac{1}{\sqrt{n_1}}V\left(\frac{n_1}{N_1}\right)\right|.
\end{equation*}
Then, using the bounds on $V(y)$ and $V'(y)$ from \eqref{eq: bounds on V}, we have
\begin{equation*}
    \Delta_2\ll kq^{1/2}H^{1/2}\sum_{n_2}\left|V\left(\frac{n_2}{N_2}\right)\right| \ll kq^{1/2}H^{1/2}N_2^{1+\epsilon}.
\end{equation*}
Therefore,
\begin{equation}\label{eq: averaged vs actual}
    \cB_3(m_1,m_2,k)=\ol{\cB_3}(m_1,m_2,k;H)+O(kq^{1-\theta+\epsilon}H^{1/2}),
\end{equation}
where
\begin{equation*}
    \ol{\cB_3}(m_1,m_2,k;H)=\frac{1}{H}\sum_{h\leq H}\cB_{3,h}(m_1,m_2,k).
\end{equation*}

We now want to analyze the sum of shifted sums. Applying Cauchy-Schwarz, we obtain
\begin{align*}
    \ol{\cB_3}(m_1,m_2,k;H)&=\frac{\sqrt{q}}{H}\sum_{h\leq H}\sum_{n_1}\sum_{n_2}\frac{1}{\sqrt{n_1n_2}}V\left(\frac{n_1}{N_1}\right)V\left(\frac{n_2}{N_2}\right)\Kl_k\left(c(n_1\ol{n_2}+h);q\right),\\
    &\leq\frac{\sqrt{q}}{H}\sum_{x\shortmod{q}}\nu(x)\left\vert\sum_{h\leq H}\Kl_k(c(x+h);q)\right\vert\\
    &\leq\frac{\sqrt{q}}{H}\left(\sum_{x\shortmod{q}} \nu^2(x)\right)^{1/2}\left(\sum_{x\shortmod{q}}\left\vert\sum_{h\leq H}\Kl_k(c(x+h);q)\right\vert^2\right)^{1/2},
\end{align*}
where
\begin{equation*}
    \nu(x)=\sum_{n_1\ol{n_2}\equiv x\shortmod{q}}\frac{1}{\sqrt{n_1n_2}}V\left(\frac{n_1}{N_1}\right)V\left(\frac{n_2}{N_2}\right).
\end{equation*}
To simplify notation, we write
\begin{equation*}
    \ol{\cB_3}(m_1,m_2,k;H)\leq\frac{\sqrt{q}}{H}V_2^{1/2}W^{1/2},
\end{equation*}
where
\begin{align*}
    V_2=\sum_{x\shortmod{q}}\nu^2(x),\quad W=\sum_{x\shortmod{q}}\left\vert\sum_{h\leq H}\Kl_k(c(x+h);q)\right\vert^2.
\end{align*}
We now produce bounds on $V_2$ and $W$.
\begin{lem}\label{lem: bounds on V2, W}
    Assume that $H\ll \sqrt{q}$. Then, using notation as above,
    \begin{align*}
        V_2&\ll q^\epsilon,\\
        W&\ll k^2Hq.
    \end{align*}
\end{lem}
\begin{proof}
    Let's first bound $V_2$. Note that
\begin{equation*}
    V_2=\sum_{n_1n_2'\equiv n_1'n_2\shortmod{q}}\frac{1}{\sqrt{n_1n_1'n_2n_2'}}V\left(\frac{n_1}{N_1}\right)V\left(\frac{n_1'}{N_1}\right)V\left(\frac{n_2}{N_2}\right)V\left(\frac{n_2'}{N_2}\right).
\end{equation*}
By symmetry, we can bound this by the same sum where $n_1n_2'\geq n_1'n_2$.
Let's first sum over $n_1,n_2'$:
\begin{equation*}
    V_2\ll\sum_{n_1',n_2}\frac{1}{\sqrt{n_1'n_2}}\left|V\left(\frac{n_1'}{N_1}\right)V\left(\frac{n_2}{N_2}\right)\right|\sum_{k\geq 0}\sum_{n_1n_2'=n_1'n_2+kq}\frac{1}{\sqrt{n_1n_2'}}\left|V\left(\frac{n_1}{N_1}\right)V\left(\frac{n_2'}{N_2}\right)\right|.
\end{equation*}
The sum along the diagonal ($k=0$) is bounded by
\begin{equation*}
    \sum_{n_1',n_2}\frac{\tau(n_1'n_2)}{n_1'n_2}\left|V\left(\frac{n_1'}{N_1}\right)V\left(\frac{n_2}{N_2}\right)\right|\ll q^\epsilon.
\end{equation*}
The sum along the off-diagonal ($k>0$) is bounded by
\begin{equation*}
    \sum_{n_1',n_2}\frac{1}{n_1'n_2}\left|V\left(\frac{n_1'}{N_1}\right)V\left(\frac{n_2}{N_2}\right)\right|\sum_{k>0}{\tau(n_1'n_2+kq)}k^{-2}\ll q^\epsilon.
\end{equation*}
For $W$, we have
\begin{equation*}
    W=\sum_{h_1,h_2\leq H}\sum_{x\shortmod{q}}\Kl_k(c(x+h_1);q)\ol{\Kl_k(c(x+h_2);q)}.\end{equation*}
Sums of products of hyper-Kloosterman sums of this type are bounded by Fouvry, Kowalski and Michel in \cite{FKM15}. Due to Katz \cite{K88}, when $k$ is even, the normalized hyper-Kloosterman sheaf $\mathcal{K}\ell_k$ is self-dual, and when $k$ is odd, its geometric monodromy group is equal to $\mathrm{SL}(r)$ as long as $q>2$. Therefore, the assumptions for \cite[Corollary 1.6]{FKM15} apply whenever $h_1\neq h_2$. Thus, tracking the conductor dependence in \cite{FKM15}, we can bound the inner sum over $x\pmod{q}$ by $k^2\sqrt{q}$. Therefore, we have the bound $W\ll k^2H^2\sqrt{q}+k^2Hq\ll k^2Hq$, as long as $H\ll\sqrt{q}$.
\end{proof}
Therefore, we obtain that
\begin{equation*}
    \ol{\cB_3}(m_1,m_2,k;H)\ll \frac{kq^{1+\epsilon}}{H^{1/2}},
\end{equation*}
and thus by \eqref{eq: averaged vs actual},
\begin{equation*}
    \cB_3(m_1,m_2,k)\ll \frac{kq^{1+\epsilon}}{H^{1/2}}+kq^{1-\theta+\epsilon}H^{1/2}.
\end{equation*}
When $H=q^\theta$, these terms are balanced, giving
\begin{equation}\label{B_3 bound}
    \cB_3(m_1,m_2,k)\ll kq^{1-\theta/2+\epsilon}.
\end{equation}

\subsection{Bounds on for the $k=1$ case}
It remains to consider the case $k=\pm 1$. Here we consider the case $k=1$, and an analogous approach will produce the same results for $k=-1$. As in \eqref{AFE sum decomp}, we have
\begin{equation}\label{eq: B1+...+B4 k=1}
    \cB(m_1,m_2,1)=\cB_1(m_1,m_2,1)+\cB_2(m_1,m_2,1)+\cB_3(m_1,m_2,1)+\cB_4(m_1,m_2,1)+O\left((X+X^{-1})q^{1/2+\epsilon}\right),
\end{equation}
but here the $\cB_1(m_1,m_2,1)$ term degenerates to
\begin{equation*}
    \cB_1(m_1,m_2,1)=\frac{\varphi(q)}{2}\underset{m_1n_1n_2\equiv\pm m_2}{\sum_{n_1}\sum_{n_2}}\frac{1}{\sqrt{n_1n_2}}V\left(\frac{n_1X}{\sqrt{q}}\right)V\left(\frac{n_2X}{\sqrt{q}}\right).
\end{equation*}
In this case, $\cB_1(m_1,m_2,1)$ cannot be bounded well enough individually. After summation over $m_1$ and $m_2$ introduced by a mollifier, we can bound these terms well on average, which we will do in 
Section \ref{sec: mol2}. If we set $X=q^{-\theta}$ with $\theta>0$, $\cB_4(m_1,m_2,1)$, is trivially bounded by $q^{1-\theta+\epsilon}$. Finally, we can bound the sums $\cB_2(m_1,m_2,1)$ and $\cB_3(m_1,m_2,1)$ using completion techniques.

Without loss of generality, we bound $\cB_3(m_1,m_2,1)$, which degenerates to
\begin{equation}\label{eq: B3}
    \cB_3(m_1,m_2,1)=\frac{\varphi(q)}{2\sqrt{q}}\sum_{n_1}\sum_{n_2}\frac{1}{\sqrt{n_1n_2}}V\left(\frac{n_1}{X\sqrt{q}}\right)V\left(\frac{n_2X}{\sqrt{q}}\right)\left[e\left(\frac{\ol{m_1}m_2n_1\ol{n_2}}{q}\right)+e\left(-\frac{\ol{m_1}m_2n_1\ol{n_2}}{q}\right)\right].
\end{equation}
We will exploit cancellation over the inverted variable, using the heuristic that the modular inverse behaves randomly mod $q$. We will first examine a simpler sum without smooth weights.
Let
\begin{equation*}
    S_m(T)=\underset{n\leq T}{\primsumempty}e\left(\frac{m\ol{n}}{q}\right).
\end{equation*}
Completing this sum using additive characters, we get
\begin{equation}\label{eq: completed sum}
    S_m(T)=\frac{1}{q}\sum_{a\shortmod{q}}\lambda\left(\frac{a}{q}\right)S(-a,m;q),
\end{equation}
where
\begin{equation*}
    \lambda\left(\frac{a}{q}\right)=\sum_{n\leq T}e\left(\frac{an}{q}\right),
\end{equation*}
and $S(-a,m;q)$ is the (non-normalized) Kloosterman sum
\begin{equation*}
    S(a,b;q)=\underset{x\shortmod{q}}{\primsumempty}e\left(\frac{ax+b\ol{x}}{q}\right).
\end{equation*}
By Weil's bound for Kloosterman sums,
\begin{equation*}
    \left|S(-a,m;q)\right|\leq 2^{\omega(q)}q^{1/2}.
\end{equation*}
Additionally, we have that
\begin{equation*}
    \left|\lambda\left(\frac{a}{q}\right)\right|\leq \frac{1}{2\|a/q\|}=\frac{q}{2|a|},
\end{equation*}
as long as we take $0<|a|\leq q/2$.
Therefore, treating the $a=0$ and $a\neq0$ terms separately in \eqref{eq: completed sum}, we obtain the bound
\begin{equation*}
    |S_m(T)|\leq \frac{T}{q}+2^{\omega(q)}\sqrt{q}\sum_{1\leq a\leq q/2}\frac{1}{a}\ll \frac{T}{q}+q^{1/2+\epsilon}
\end{equation*}
for $q$ prime.
Also noting the trivial bound for $S_m(T)$, we have
\begin{equation*}
    S_m(T)\ll\min\left(T,\frac{T}{q}+q^{1/2+\epsilon}\right).
\end{equation*}
Now to analyze the sum with weights coming from $\cB_3(m_1,m_2,1)$, we may use partial summation. The inner sums over $n_2$ in  \eqref{eq: B3} reduce to
\begin{align}\label{partial summation integral}
\sum_{n_2}\frac{1}{\sqrt{n_2}}V\left(\frac{n_2X}{\sqrt{q}}\right)e\left(\frac{m\ol{n_2}}{q}\right)=-\int_1^\infty \left(-\frac{1}{2t^{3/2}}V\left(\frac{tX}{\sqrt{q}}\right)+\frac{1}{t^{1/2}}\frac{X}{\sqrt{q}}V'\left(\frac{tX}{\sqrt{q}}\right)\right)S_m(t)\,dt,
\end{align}
where $m=\pm\ol{m_1}m_2n_1$.
Splitting \eqref{partial summation integral} into $I_1=\int_1^{\sqrt{q}}$ and $I_2=\int_{\sqrt{q}}^\infty$ and using that $V(y),\,yV'(y)\ll(1+y)^{-1}$ from \eqref{eq: bounds on V},
\begin{align*}
    I_1&\ll\int_1^{\sqrt{q}}\frac{1}{t^{3/2}}\left(1+\frac{tX}{\sqrt{q}}\right)^{-1}t\,dt\ll q^{1/4},\\
    I_2&\ll \int_{\sqrt{q}}^\infty \frac{1}{t^{3/2}}\left(1+\frac{tX}{\sqrt{q}}\right)^{-1}\left(\frac{t}{q}+q^{1/2+\epsilon}\right)\ll q^{1/4+\epsilon}.
\end{align*}
Therefore, using the trivial bound for the sum over $n_1$,
\begin{equation*}
    \cB_3(m_1,m_2,1)\ll q^{3/4+\epsilon}(X\sqrt{q})^{1/2}= q^{1-\theta/2+\epsilon},
\end{equation*}
and an identical bound holds for $\cB_2(m_1,m_2,k)$.
Applying these bounds to \eqref{eq: B1+...+B4 k=1},
\begin{equation}\label{B for k=1}
    \cB(m_1,m_2,1)=\cB_1(m_1,m_2,1)+O(q^{1-\theta/2+\epsilon}),
\end{equation}
as long as $0<\theta<1/3$.

\section{Mollified moments}
For the following sections, we use the mollifer introduced in \cite{IS99}, which takes the form
\begin{equation}\label{eq: mollifier equation}
M(\chi)=\sum_{m\leq M}\frac{x_m\chi(m)}{m^{1/2}},
\end{equation}
where
\begin{equation}\label{mollifier coefficients}
x_m=\frac{\mu(m)m}{\varphi(m)}\frac{1}{G}\sum_{\substack{\ell\leq M/m\\(\ell,mq)=1}}\frac{\mu^2(\ell)}{\varphi(\ell)},
\end{equation}
supported on $(m,q)=1$, and $G$ is given by the sum
\begin{equation}\label{eq: defn of G}
G=\sum_{\substack{k\leq M\\ (k,q)=1}}\frac{\mu^2(k)}{\varphi(k)}.
\end{equation}
\subsection{Mollified first moments}\label{sec: mol1}
Using this mollifier, we prove a bound on the mollified first moment twisted by powers of the root number.
\begin{prop}\label{prop: first mollified moment}
Let $q$ be a positive integer, and the mollifier $M(\chi)$ as defined above with length $M=q^\alpha$ and $0<\alpha<1/2$. Then we have
    \begin{align*}
    \evensum \epsilon(\chi)^kM(\chi) L(1/2,\chi)=
    \begin{cases}
        \displaystyle\varphi^+(q)+O\left(\tau (q) q^{1/2+\alpha}\right),\quad &k=0,\\[5pt]
        \displaystyle O\left(\frac{q}{\alpha\log q}\right),\quad &k=-1,\\[10pt]
        \displaystyle O\left(|k|^{\omega(q)}\tau(q)q^{3/4+\alpha/2}\right),\quad &k\neq0,-1.
    \end{cases}
\end{align*}
\end{prop}
The choice of mollifier coefficients \eqref{mollifier coefficients} satisfies the conditions that $x_1=1$ and $|x_m|\leq 1$ for all $m$. Additionally, evaluating the sum \eqref{eq: defn of G} yields
\begin{equation*}
    G=\frac{\varphi(q)}{q}
\left(\log M+c+\sum_{p\mid q}\frac{\log p}{p}\right)+O\left(\frac{\theta(q)}{\sqrt{M}}\right),
\end{equation*}
where $c$ is an absolute constant, and $\theta(q)=2^{\omega(q)}$. The details of this evaluation are provided in Appendix \ref{section: sum}. Since $M=q^\alpha$, we have the lower bound $G\gg \frac{\varphi(q)}{q}\log M$.

The mollifed first moment twisted by $\epsilon(\chi)^k$ is given by
\begin{equation*}
    \cC(k):=\evensum \epsilon(\chi)^kM(\chi) L(1/2,\chi)=\sum_{m\leq M}\frac{x_m}{\sqrt{m}}\cA(m,k).
\end{equation*}
By the bounds on $\cA(m,k)$ from Theorem \ref{thm: first weighted moment} and the bound $|x_m|\leq 1$, we obtain the results in Proposition \ref{prop: first mollified moment} for $k\neq -1$. The sum $\cC(-1)$ contributes a secondary main term of
\begin{align*}
    \varphi^+(q)\sum_{m\leq M}\frac{x_m}{m}=\frac{\varphi^+(q)}{G}\underset{\substack{m\ell\leq M\\ m,l,q\textrm{ pairwise coprime}}}{\sum\sum}\frac{\mu(m)\mu^2(\ell)}{\varphi(m)\varphi(\ell)}.
\end{align*}
This double sum has the structure
\begin{equation*}
    \sum_{n\leq M}(f\star_1 g)(n),
\end{equation*}
with
\begin{equation*}
    f(n)=\frac{\mu(n)}{\varphi(n)}\mathds{1}_{(n,q)=1},\quad g(n)=\frac{\mu^2(n)}{\varphi(n)}\mathds{1}_{(n,q)=1},
\end{equation*}
where $f\star_1 g$ denotes the unitary convolution
\begin{equation*}
    (f\star_1 g)(n)=\sum_{\substack{lm=n\\(l,m)=1}}f(l)g(m).
\end{equation*}
The unitary convolution preserves multiplicativity, and checking on prime powers gives that $(f\star_1 g)(n)=\delta(n)$. Therefore, the main contribution from $\cC(-1)$ is $\frac{\varphi^+(q)}{G}=O(q/\log M)$. The error term of $\cA(m,-1)$ is easily bounded by this main contribution, proving Proposition \ref{prop: first mollified moment}.

\subsection{Mollified second moment}\label{sec: mol2}
In this section, we evaluate the mollified second moments, twisted by powers of the root number.
\begin{prop}\label{prop: second mollified moment}
Let $q$ be a prime and $M(\chi)$ the mollifier as defined in \eqref{eq: mollifier equation}--\eqref{eq: defn of G} with length $M=q^\alpha$ and $0<\alpha<1/24$. Then we have
\begin{align*}
    \evensum \epsilon(\chi)^k|M(\chi)L(1/2,\chi)|^2=
    \begin{cases}
        \displaystyle \left(1+\frac{1}{\alpha}\right)\varphi^+(q)+O\left(\frac{\varphi^+(q)\log\log q}{\alpha\log q}\right),\quad &k=0,\\[10pt]
        \displaystyle O\left(\frac{\varphi^+(q)}{\alpha\log q}\right),\quad&k=\pm1,\\[10pt]
        \displaystyle O_{\epsilon}\left(|k|^{18}q^{23/24+\alpha+\epsilon}\right),\quad &|k|\geq 2.
    \end{cases}
\end{align*}
\end{prop}
The case $k=0$ is treated in \cite{IS99}.
Now, we twist by $\epsilon(\chi)^k$ for $k\neq 0$. Let's first consider the case of $|k|\geq2$. We want to evaluate
\begin{equation*}
    \cD(k):=\evensum \epsilon(\chi)^k|M(\chi)L(1/2,\chi)|^2=\sum_{m_1\leq M}\sum_{m_2\leq M}\frac{x_{m_1}x_{m_2}}{\sqrt{m_1m_2}}\cB(m_1,m_2,k).
\end{equation*}
Using the bounds on $\cB(m_1,m_2,k)$ from Theorem \ref{thm: second weighted moment} and the bound $|x_m|\leq 1$, we obtain the results in Proposition \ref{prop: second mollified moment} for $|k|\geq 2$.

The only moments that still need to be estimated are the mollified second moments twisted by $\epsilon(\chi)^{\pm1}$. Without loss of generality, we will consider the $k=1$ term, where $\cD(1)$ appears to have a nontrivial contribution coming from $\cB_1(m_1,m_2,1)$. Using \eqref{B for k=1}, we have
\begin{equation}\label{secondary sum}
\cD(1)=\varphi^+(q)\sum_{m_1\leq M}\sum_{m_2\leq M}\frac{x_{m_1}x_{m_2}}{\sqrt{m_1m_2}}\underset{m_1n_1n_2\equiv \pm m_2}{\sum_{n_1}\sum_{n_2}}\frac{1}{\sqrt{n_1n_2}}V\left(\frac{n_1X}{\sqrt{q}}\right)V\left(\frac{n_2X}{\sqrt{q}}\right)+O\left(q^{1-\theta/2+\alpha+\epsilon}\right),
\end{equation}
where $X=q^{-\theta}$.
The off-diagonal terms (where $m_1n_1n_2\neq m_2$) can crudely be bounded by $M^{5/2}q^{1/2+2\theta+\epsilon}$. Taking $\theta=7/40$ balances the two error terms. Decomposing \eqref{secondary sum} into diagonal and off-diagonal terms yields
\begin{align*}
\cD(1)=\varphi^+(q)\underset{mm_1\leq M}{\sum\sum}\frac{x_{m_1}x_{mm_1}}{mm_1}\tau(m)\left(1+O\left(\frac{m^{1/2}X^{1/2}}{q^{1/4}}\right)\right)+O(q^{229/240+\epsilon}).
\end{align*}
Using $|x_m|\leq 1$, the error term in the sum absorbs into $O(q^{229/240})$. For the remaining sum, we use the definition of $x_m$ from \eqref{mollifier coefficients}. This gives a main term of
\begin{equation*}
    S:=\frac{\varphi^+(q)}{G^2}\underset{\substack{mm_1\leq M\\ m,m_1,q\textrm{ pairwise coprime}}}{\sum\sum}\frac{\mu^2(m_1)\mu(m)m_1\tau(m)}{\varphi^2(m_1)\varphi(m)}\left(\sum_{\substack{\ell_1\leq M/m_1\\ (\ell_1,m_1q)=1}}\frac{\mu^2(\ell_1)}{\varphi(\ell_1)}\right)\left(\sum_{\substack{\ell_2\leq M/mm_1\\ (\ell_2,mm_1q)=1}}\frac{\mu^2(\ell_2)}{\varphi(\ell_2)}\right).
\end{equation*}
Using Proposition \ref{prop: eval of G} to evaluate the sum over $\ell_1$, we have
\begin{align*}
    S=\frac{\varphi^+(q)}{G^2}\underset{\substack{mm_1\ell\leq M\\ m,m_1,\ell,q\textrm{ pairwise coprime}}}{\sum\sum\sum}&\frac{\mu^2(m_1)\mu(m)m_1\tau(m)\mu^2(\ell)}{\varphi^2(m_1)\varphi(m)\varphi(\ell)}\frac{\varphi(m_1)\varphi(q)}{m_1 q}\\
    &\cdot\left(\log\frac{M}{m_1}+c+\sum_{p\mid q}\frac{\log p}{p}+\sum_{p\mid m_1}\frac{\log p}{p}+O\left(\frac{\theta(q)\sqrt{m_1}}{\sqrt{M}}\right)\right).
\end{align*}
Write $S=S_0+R$, where
\begin{align*}
    S_0=\frac{\varphi^+(q)}{G^2}\underset{\substack{mm_1\ell\leq M\\ m,m_1,\ell,q\textrm{ pairwise coprime}}}{\sum\sum\sum}&\frac{\mu^2(m_1)\mu(m)m_1\tau(m)\mu^2(\ell)}{\varphi^2(m_1)\varphi(m)\varphi(\ell)}\frac{\varphi(m_1)\varphi(q)}{m_1 q}\nonumber\\
    &\cdot\left(\log\frac{M}{m_1}+c+\sum_{p\mid q}\frac{\log p}{p}+\sum_{p\mid m_1}\frac{\log p}{p}\right),
\end{align*}
and
\begin{align}\label{eq: remainder term}
    R=\frac{\varphi^+(q)}{G^2}\underset{\substack{mm_1\ell\leq M\\ m,m_1,\ell,q\textrm{ pairwise coprime}}}{\sum\sum\sum}&\frac{\mu^2(m_1)\mu(m)m_1\tau(m)\mu^2(\ell)}{\varphi^2(m_1)\varphi(m)\varphi(\ell)}\frac{\varphi(m_1)\varphi(q)}{m_1 q}\cdot O\left(\frac{\theta(q)\sqrt{m_1}}{\sqrt{M}}\right).
\end{align}
Since $q$ is prime, $\sum_{p\mid q}\frac{\log p}{p}=\frac{\log q}{q}$.
Additionally, $m_1$ is forced to be squarefree, yielding
\begin{align}
    S_0=\frac{\varphi^+(q)\varphi(q)}{qG^2}\underset{\substack{mm_1\ell\leq M\\ m,m_1,\ell,q\textrm{ pairwise coprime}}}{\sum\sum\sum}&\frac{\mu^2(m_1)\mu(m)\tau(m)\mu^2(\ell)}{\varphi(m_1)\varphi(m)\varphi(\ell)}\label{eq: S_0}\\
    &\cdot\left(\log M+c+\frac{\log q}{q}+\sum_{p\mid m_1}\left(\frac{(1-p)\log p}{p}\right)\right).\nonumber
\end{align}
First consider the part of \eqref{eq: S_0} not involving the sum over $p\mid m_1$, which we will call $S_1$. We have
\begin{align*}
    S_1=\frac{\varphi^+(q)\varphi(q)}{qG^2}\left(\log M+c+\frac{\log q}{q}\right)\underset{\substack{mm_1\ell\leq M\\ m,m_1,\ell,q\textrm{ pairwise coprime}}}{\sum\sum\sum}&\frac{\mu^2(m_1)\mu(m)\tau(m)\mu^2(\ell)}{\varphi(m_1)\varphi(m)\varphi(\ell)}.
\end{align*}
The three variable sum in $S_1$ can be written as
\begin{equation*}
    \sum_{n\leq M}(f\star_1 g\star_1 g)(n),
\end{equation*}
where
\begin{align*}
    f(n)=\frac{\mu(n)\tau(n)}{\varphi(n)}\mathds{1}_{(n,q)=1},\quad g(n)=\frac{\mu^2(n)}{\varphi(n)}\mathds{1}_{(n,q)=1}.
\end{align*}
One can check that $(f\star_1 g\star_1 g)(n)=\delta(n)$, which implies that
\begin{equation*}
    S_1=\frac{\varphi^+(q)\varphi(q)}{qG^2}\left(\log M+c+\frac{\log q}{q}\right)\ll \frac{\varphi^+(q)}{\log M}.
\end{equation*}
We can rewrite the second part of \eqref{eq: S_0} as
\begin{equation*}
    S_2=\frac{\varphi^+(q)\varphi(q)}{qG^2}\sum_{p\leq M}\left(\frac{(1-p)\log p}{p}\right)\underset{\substack{mm_1\ell\leq M/p\\ m,m_1,\ell,q\textrm{ pairwise coprime}\\(m\ell q,p)=1}}{\sum\sum\sum}\frac{\mu^2(pm_1)\mu(m)\tau(m)\mu^2(\ell)}{\varphi(pm_1)\varphi(m)\varphi(\ell)}.
\end{equation*}
The factor $\mu^2(pm_1)$ enforces that $(p,m_1)=1$, so we have
\begin{equation*}
    S_2=\frac{\varphi^+(q)\varphi(q)}{qG^2}\sum_{p\leq M}\left(\frac{(1-p)\log p}{p(p-1)}\right)\underset{\substack{mm_1\ell\leq M/p\\ m,m_1,\ell,p,q\textrm{ pairwise coprime}}}{\sum\sum\sum}\frac{\mu^2(m_1)\mu(m)\tau(m)\mu^2(\ell)}{\varphi(m_1)\varphi(m)\varphi(\ell)}.
\end{equation*}
As before, the three variable sum can be viewed as a one-variable sum of a unitary convolution which evaluates to $\delta(n)$. Therefore, we get a bound
\begin{equation*}
    S_2=-\frac{\varphi^+(q)\varphi(q)}{qG^2}\sum_{p\leq M}\frac{\log p}{p}\ll\frac{\varphi^+(q)}{\log M},
\end{equation*}
using Mertens' theorem to handle the sum over primes. For the error term \eqref{eq: remainder term}, we have a bound of
\begin{equation*}
    R\ll\frac{\varphi^+(q)\theta(q)}{\sqrt{M}\log^2M}\sum_{n\leq M}a(n),
\end{equation*}
where $a(n)$ is a multiplicative function supported on squarefree integers satisfying
\begin{equation*}
    a(p)=\frac{\sqrt{p}+3}{p-1}.
\end{equation*}
We note that the partial sums of $a(n)$ are bounded by

    \begin{equation}\label{a(n) partial sums}
    \sum_{n\leq M}a(n)\ll\sqrt{M},
\end{equation}
obtained by applying Mellin inversion to a smoothed version of this sum.
With the bound \eqref{a(n) partial sums}, we get that the nontrivial contribution to the mollified second moment is given by
\begin{equation*}
    \cD_1(1)\ll \varphi^+(q)\left(\frac{1}{\log M}+\frac{\theta(q)}{(\log M)^{2}}\right).
\end{equation*}
For $q$ prime, $\theta(q)=2$, and thus $\cD_1(1)=O(\varphi^+(q)/\log M)$.

\subsection{Smoothed moments}
We can use the results in sections \ref{sec: mol1} and \ref{sec: mol2} to prove Theorem \ref{thm: smoothed mollified moments}, the mollified first and second moments weighted by a smooth function of the root angle, $\theta_\chi$. We may consider a family of functions $f_q:\R/\Z\to\R$, varying over $q$ prime. Let $A$ denote the exponent of the polynomial dependence on $|k|$ in Proposition \ref{prop: second mollified moment}, which we have shown is at most $18$. Fix an integer $J>\max\{2,A+1\}$, and suppose that conditions \eqref{f_q derivative condition} and \eqref{f_q condition 2} are satisfied.

Suppose that $f_q(x)$ has Fourier expansion
\begin{equation*}
    f_q(x)=\sum_{k\in\Z}c_{k,q}e(kx).
\end{equation*}
Then,
\begin{equation*}
    \cC:=\evensum f_q(\theta_\chi)M(\chi)L(1/2,\chi)=\sum_{k\in\Z}c_{k,q}\evensum M(\chi)L(1/2,\chi)\epsilon(\chi)^k.
\end{equation*}
The term $k=0$ gives the main term in Theorem \ref{thm: smoothed mollified moments} and condition \eqref{f_q condition 2} ensures that the $k=-1$ term is absorbed into $o\left(\varphi^+(q)\right)$. Then, for $k\neq0$, we have the bound on Fourier coefficients
\begin{equation}\label{eq: fourier coeff bound}
    |c_{k,q}|\ll_J \frac{1}{|k|^J}\|f_q^{(J)}\|
\end{equation}
by integration by parts. Since $J>2$, the sum over $k\neq0,-1$ converges, and due to condition \eqref{f_q derivative condition}, can be absorbed into $o\left(\varphi^+(q)\right)$.

For the second moment, the procedure is similar. Write
\begin{equation*}
    \cD=\evensum f_q(\theta_\chi)|M(\chi)L(1/2,\chi)|^2.
\end{equation*}
Just as in the first moment, the $k=0$ term gives the main term in Theorem \ref{thm: smoothed mollified moments} and condition \eqref{f_q condition 2} guarantees that the $k=\pm1$ terms absorb into $o\left(\varphi^+(q)\right)$. The sum over $|k|\geq 2$ is bounded by
\begin{equation*}
    \sum_{|k|\geq 2}|c_{k,q}|\left|\evensum \epsilon(\chi)^k\left|M(\chi)L(1/2,\chi)\right|^2\right|\ll q^{23/24+\alpha+\epsilon}\sum_{|k|\geq 2}|c_{k,q}|\cdot|k|^A,
\end{equation*}
by Proposition \ref{prop: second mollified moment}. Then by the bound \eqref{eq: fourier coeff bound} and condition \eqref{f_q derivative condition}, the sum over $k$ converges and absorbs into $o\left(\varphi^+(q)\right)$. This completes the proof of Theorem \ref{thm: smoothed mollified moments}.

\section{Non-Vanishing}
Now, we want to prove Theorem \ref{thm: positive proportion}, which requires that a positive proportion of even primitive Dirichlet $L$-functions with $\theta_\chi$ in some interval satisfy
    \begin{equation}\label{central point lb}
        |L(1/2,\chi)|>\delta\cdot(\log q)^{-1/2},
    \end{equation}
    for some $\delta>0$.
Let $I$ be an interval of positive length. To prove our non-vanishing result, we will take $f:\R/\Z\to\R$ to minorize the indicator function $\chi_I$.
Let $\cH_{\textrm{bad}}$ be the set of characters for which \eqref{central point lb} is not satisfied, and $\cH_{\textrm{good}}$ be the set for which \eqref{central point lb} is satisfied. Likewise, we split our sum $\cC=\cC_{\textrm{bad}}+\cC_{\textrm{good}}$.

\begin{align*}
    \cC_{\textrm{bad}}&\leq \delta\cdot(\log q)^{-1/2}\evensum |f(\theta_\chi)M(\chi)|\\
    &\leq\delta\cdot(\log q)^{-1/2}\left(\evensum |f(\theta_\chi)|^2\right)^{1/2}\left(\evensum |M(\chi)|^2\right)^{1/2}\\
    &\leq\delta\cdot(\log q)^{-1/2}\left(\varphi^+(q)\evensum |M(\chi)|^2\right)^{1/2}\leq 2\alpha\delta\varphi^+(q).
\end{align*}
By Theorem \ref{thm: smoothed mollified moments}, for $q$ sufficiently large (depending on $\alpha$ and $f$),
\begin{equation*}
\cC_\textrm{good}>(\|f\|-3\alpha\delta)\varphi^+(q).
\end{equation*}
But we also have that
\begin{equation*}
    |\cC_{\textrm{good}}|^2\leq \cD\underset{\chi\in\cH_{\textrm{good}}}{\evensum}f(\theta_\chi),
\end{equation*}
as long as $f\geq0$. Therefore, using the bounds
\begin{align}
    \underset{\chi\in\cH_{\textrm{good}}}{\evensum}f(\theta_\chi)&\geq\frac{|\cC_{\textrm{good}}|^2}{\cD}\label{eq: lower bound}\\
    &>\frac{\varphi^+(q)}{1+\alpha^{-1}}\frac{(\|f\|-3\alpha\delta)^2}{\|f\|(1+\alpha\delta)}\nonumber\\
    &>\frac{\varphi^+(q)}{1+\alpha^{-1}}(\|f\|-7\alpha\delta)\nonumber
\end{align}
for $q$ sufficiently large.
We now construct an appropriate bump function $f$ to obtain a positive proportion of non-vanishing. Let $g:[0,1]\to\R_{\geq0}$ be a smooth function satisfying $0\leq g(x)\leq 1$, $g(x)=g(1-x)$, $g(0)=0$, $g\left(\frac{1}{2}\right)=1$, and $g^{(j)}(0)=g^{(j)}\left(\frac{1}{2}\right)=0$ for all $j>0$. We then construct $g_\beta(x)$ for $\beta\leq 1$ by taking
\begin{align*}
    g_\beta(x)=
    \begin{cases}
        g\left(\frac{x}{\beta}\right),\quad &0\leq x<\frac{\beta}{2},\\
        1,\quad &\frac{\beta}{2}\leq x\leq 1-\frac{\beta}{2},\\
        g\left(\frac{1-x}{\beta}\right),\quad &1-\frac{\beta}{2}<x\leq 1.
    \end{cases}
\end{align*}
The function $g_\beta(x)$ is smooth and $\|g_\beta\|_1\geq 1-\beta$.
Now, let $I=(a,b)$ with $a,b\in\R$ and $b-a\leq1$ Define
\begin{align*}
    f_{\beta,I}(x)=
    \begin{cases}
        g_\alpha\left(\frac{x-a}{b-a}\right),\quad &a\leq x\leq b,\\
        0,\quad &b<x<a+1,
    \end{cases}
\end{align*}
and extend to a smooth function $f_{\beta,I}:\R/\Z\to\R$ by periodicity. This function has $L^1$ norm $\|f_{\beta,I}\|\geq (1-\beta)\mu(I)$ and minorizes the indicator function $\chi_I$.

To obtain a positive proportion of non-vanishing for shrinking intervals, we take $f_q=f_{\beta,I_q}$, where the interval depends on $q$. For the constant $J$ as in condition \eqref{f_q derivative condition}, we have that
\begin{equation*}
    \|f_{\beta,I_q}^{(J)}\|\ll_J \frac{1}{\left(\beta\cdot\mu(I_q)\right)^{J-1}},
\end{equation*}
and thus
\begin{equation*}
    \|f_{\beta,I_q}^{(J)}\|\left|\int_0^1f_{\beta,I_q}(x)\,dx\right|^{-1}\ll_{\beta,J}\frac{1}{\mu(I_q)^J}.
\end{equation*}
If $\mu(I_q)\gg q^{-\eta}$, where $\eta<\frac{1}{24J}$, we satisfy condition \eqref{f_q derivative condition} for $\alpha<\frac{1}{24}-J\eta$. Condition \eqref{f_q condition 2} is trivially satisfied by positivity of $f_{\beta,I_q}$.  We take $\delta=\mu(I_q)\epsilon$, $\alpha=(\frac{1}{24}-J\eta)(1-\frac{\epsilon}{3})$ and $\beta=\frac{\epsilon}{3}$. As discussed earlier in this section, we may take $J=20$. Applying these choices to \eqref{eq: lower bound} yields Theorem \ref{thm: positive proportion}.

\appendix

\section{Evaluation of an arithmetic sum}\label{section: sum}
In this section, we evaluate an asymptotic for the sum
\begin{equation*}
G=\sum_{\substack{n\leq M\\ (n,q)=1}}\frac{\mu^2(n)}{\varphi(n)},
\end{equation*}
which appears in the coefficients of the mollifier $M(\chi)$.
\begin{prop}\label{prop: eval of G}
    Let $q$ be a positive integer. Then,
    \begin{equation*}
        \sum_{\substack{n\leq M\\(n,q)=1}}\frac{\mu^2(n)}{\varphi(n)}=\frac{\varphi(q)}{q}\left(\log M+c+\sum_{p\mid q}\frac{\log p}{p}\right)+O\left(\frac{\theta(q)}{\sqrt{M}}\right),
    \end{equation*}
    where
    \begin{equation*}
        c=\gamma+\sum_p\frac{\log{p}}{p(p-1)}
    \end{equation*}
    is an absolute constant.
\end{prop}
Our evaluation of this sum follows the approach of R. Sitaramachandra Rao \cite{Sit85} who considers a similar sum without the coprimality restriction.
One can easily show that
    \begin{equation*}
    \sum_{\substack{n\leq M\\ (n,q)=1}}\mu^2(n)\frac{n}{\varphi(n)}=\frac{\varphi(q)}{q}M+O(\sqrt{M}\theta(q)),
    \end{equation*}
where $\theta(q)=2^{\omega(q)}$.
If we express this sum as 
\begin{equation*}
    \sum_{\substack{n\leq M\\ (n,q)=1}}\mu^2(n)\frac{n}{\varphi(n)}=\frac{\varphi(q)}{q}M+\Delta(M),
\end{equation*}
where $\Delta(M)=O(\sqrt{M}\theta(q))$,
then by partial summation we have that
\begin{equation*}
    \sum_{{ (n,q)=1}}\frac{\mu^2(n)n}{\varphi(n)}\frac{1}{n^s}=\frac{\varphi(q)}{q}\frac{s}{s-1}+s\int_1^\infty \frac{\Delta(t)}{t^{s+1}}\,dt.
\end{equation*}
By Euler's product theorem, we also have that
\begin{align*}
    \sum_{(n,q)=1}\frac{\mu^2(n)n}{\varphi(n)}\frac{1}{n^s}&=\prod_p\left(1+\frac{1}{(p-1)p^{s-1}}\right)\prod_{p\mid q}\left(1+\frac{1}{(p-1)p^{s-1}}\right)^{-1}\\
    &=\zeta(s)h(s)\rho_q(s),
\end{align*}
where
\begin{equation*}
    h(s)=\prod_p\left(1+\frac{1}{(p-1)p^s}-\frac{1}{(p-1)p^{2s-1}}\right)
\end{equation*} and
\begin{equation*}
    \rho_q(s)=\prod_{p\mid q}\left(1+\frac{1}{(p-1)p^{s-1}}\right)^{-1}.
\end{equation*}
Equating these two expressions for the sum, we get that
\begin{equation*}
    \int_1^\infty \frac{\Delta(t)}{t^{s+1}}\,dt=-\frac{\varphi(q)}{q}\frac{1}{s-1}+\frac{1}{s}\zeta(s)h(s)\rho_q(s).
\end{equation*}
The product $h(s)$ converges and is holomorphic for $\Re(s)>1/2$, with $h(1)=1$, and $\rho_q(1)=\frac{\varphi(q)}{q}$. Therefore, the poles at $s=1$ cancel out and we have a holomorphic function for $\Re(s)>1/2$. We can evaluate the value at $s=1$ by looking at the constant term of the Laurent expansion of $\frac{1}{s}\zeta(s)h(s)\rho_q(s)$,
\begin{align*}
    \frac{1}{s}\zeta(s)h(s)\rho_q(s)&=\left(1-(s-1)+\cdots\right)\times\frac{1}{s-1}\left(1+\gamma(s-1)+\cdots\right)\times\left(1+\sum_p\frac{\log p}{p(p-1)}(s-1)+\cdots\right)\\&\times\frac{\varphi(q)}{q}\left(1+\sum_{p\mid q}\frac{\log p}{p}(s-1)+\cdots\right).
\end{align*}
We thus get that
\begin{equation*}
    \int_1^\infty \frac{\Delta(t)}{t^{s+1}}\,dt=\frac{\varphi(q)}{q}\left(-1+\gamma+\sum_p\frac{\log p}{p(p-1)}+\sum_{p\mid q}\frac{\log p}{p}\right)
\end{equation*}
Now, by partial summation,
\begin{align*}
    \sum_{\substack{n\leq M\\(n,q)=1}}\frac{\mu^2(n)}{\varphi(n)}&=\left(\frac{\varphi(q)}{q}M+\Delta(M)\right)\frac{1}{M}+\frac{\varphi(q)}{q}\log M+\int_1^M\frac{\Delta(t)}{t^2}\,dt\\
    &=\frac{\varphi(q)}{q}\left(\log M+\gamma+\sum_p\frac{\log p}{p(p-1)}+\sum_{p\mid q}\frac{\log p}{p}\right)+O\left(\frac{\Delta(M)}{M}+\int_M^\infty \frac{\Delta(t)}{t^2}\right)\\
    &=\frac{\varphi(q)}{q}\left(\log M+\gamma+\sum_p\frac{\log p}{p(p-1)}+\sum_{p\mid q}\frac{\log p}{p}\right)+O\left(\frac{\theta(q)}{\sqrt{M}}\right).
\end{align*}
This completes the proof of Proposition \ref{prop: eval of G}.

\bibliographystyle{alpha}
\bibliography{refs}

\end{document}

%% file: preamble.tex
\usepackage{amssymb,amsfonts,amsmath}
\usepackage[usenames,dvipsnames]{color}
\usepackage{amssymb}
\usepackage[normalem]{ulem}
\usepackage{url}
\usepackage[all,arc,2cell]{xy}
\UseAllTwocells
\usepackage{enumerate}
\usepackage{amsthm}
\usepackage{quiver}
\usepackage[letterpaper,margin=1in]{geometry}
\usepackage{bbm}
\usepackage{dsfont}
\usepackage{mathrsfs}
\usepackage{hyperref}  
\hypersetup{%
  bookmarksnumbered=true,%
  colorlinks=true,%
  linkcolor=blue,%
  citecolor=blue,%
  filecolor=blue,%
  menucolor=blue,%
  urlcolor=blue,%
  pdfnewwindow=true,%
  pdfstartview=FitBH}   

\def\makeautorefname#1#2{\expandafter\def\csname#1autorefname\endcsname{#2}}

\def\equationautorefname~#1\null{(#1)\null}
\makeautorefname{footnote}{footnote}%
\makeautorefname{item}{item}%
\makeautorefname{figure}{Figure}%
\makeautorefname{table}{Table}%
\makeautorefname{part}{Part}%
\makeautorefname{appendix}{Appendix}%
\makeautorefname{chapter}{Chapter}%
\makeautorefname{section}{Section}%
\makeautorefname{subsection}{Section}%
\makeautorefname{subsubsection}{Section}%
\makeautorefname{theorem}{Theorem}%
\makeautorefname{thm}{Theorem}%
\makeautorefname{cor}{Corollary}%
\makeautorefname{lem}{Lemma}%
\makeautorefname{prop}{Proposition}%
\makeautorefname{pro}{Property}
\makeautorefname{conj}{Conjecture}%
\makeautorefname{defn}{Definition}%
\makeautorefname{notn}{Notation}
\makeautorefname{notns}{Notations}
\makeautorefname{rmk}{Remark}%
\makeautorefname{quest}{Question}%
\makeautorefname{exmp}{Example}%
\makeautorefname{exer}{Exercise}%
\makeautorefname{ax}{Axiom}%
\makeautorefname{claim}{Claim}%
\makeautorefname{ass}{Assumption}%
\makeautorefname{asss}{Assumptions}%
\makeautorefname{con}{Construction}%
\makeautorefname{prob}{Problem}%
\makeautorefname{warn}{Warning}%
\makeautorefname{obs}{Observation}%
\makeautorefname{conv}{Convention}%

\newtheorem{thm}{Theorem}[section]

\newtheorem{prop}[thm]{Proposition}
\newtheorem{lem}[thm]{Lemma}

\theoremstyle{definition}

\newtheorem{rmk}[thm]{Remark}

\makeatletter
\let\c@obs=\c@thm
\let\c@cor=\c@thm
\let\c@prop=\c@thm
\let\c@lem=\c@thm
\let\c@prob=\c@thm
\let\c@con=\c@thm
\let\c@conj=\c@thm
\let\c@def=\c@thm
\let\c@notn=\c@thm
\let\c@notns=\c@thm
\let\c@exmp=\c@thm
\let\c@ax=\c@thm
\let\c@pro=\c@thm
\let\c@ass=\c@thm
\let\c@warn=\c@thm
\let\c@rmk=\c@thm
\let\c@sch=\c@thm
\numberwithin{equation}{section}
\makeatother

\newcommand{\ol}{\overline}

\newcommand{\cA}{\mathcal{A}}
\newcommand{\cB}{\mathcal{B}}
\newcommand{\cC}{\mathcal{C}}
\newcommand{\cD}{\mathcal{D}}
\newcommand{\cF}{\mathcal{F}}

\newcommand{\cH}{\mathcal{H}}

\newcommand{\R}{\mathbb{R}}

\newcommand{\C}{\mathbb{C}}
\newcommand{\Z}{\mathbb{Z}}

\newcommand{\Kl}{\mathrm{Kl}}
\newcommand{\cond}{\mathrm{cond}}

\newcommand{\shortmod}{\ensuremath{\negthickspace \negthickspace \negthickspace \pmod}}

\NewDocumentCommand{\rsum}{ O{\null} O{\null} m }{\sum_{#1}^{#2}{{\vphantom{\sum}}^{\negthickspace\negthickspace\negthickspace#3}}}
\NewDocumentCommand{\rprod}{ O{\null} O{\null} m }{\sum_{#1}^{#2}{{\vphantom{\prod}}^{#3}}}

\newcommand{\primsum}{\sum_{\chi\shortmod{q}}{\vphantom{\sum}}^{\negthickspace \negthickspace \negthickspace \negthickspace \negthickspace *}}

\newcommand{\primsumempty}{\sum{\vphantom{\sum}}^{\negthickspace *}}

\newcommand{\evensumrest}{\sum_{\substack{\chi\shortmod{q}\\ \theta_\chi\in I}}{\vphantom{\sum}}^{\negthickspace \negthickspace \negthickspace \negthickspace \negthickspace +}}

\newcommand{\evensum}{\sum_{\chi\shortmod{q}}{\vphantom{\sum}}^{\negthickspace \negthickspace \negthickspace \negthickspace \negthickspace +}}

\renewcommand{\Re}{\mathrm{Re}}